\documentclass[a4paper, 11pt]{amsart}
\usepackage {amsmath, amscd}
\usepackage {amssymb}
\usepackage {bm}

\numberwithin{equation}{section}

\def\<{\langle}
\def\>{\rangle}

\def\DD{{\mathcal D}}
\def\EE{{\mathcal E}}
\def\FF{{\mathcal F}}

\def\HH{{\mathcal H}}

\def\LL{{\mathcal L}}

\def\bbC{\mathbb{C}}

\def\bbD{\mathbb{D}}
\def\bbT{\mathbb{T}}

\newtheorem{lemma}{Lemma}[section]
\newtheorem{theorem}[lemma]{Theorem}
\newtheorem{corollary}[lemma]{Corollary}

\theoremstyle{definition}
\newtheorem{remark}[lemma]{Remark}

\newtheorem{example}[lemma]{Example}

\title{The invariant subspaces of $ S\oplus S^* $}
\author{Dan Timotin}
\address{Institute of Mathematics Simion Stoilow of the Romanian Academy, Calea Grivi\c tei 21, Bucharest 010702, Romania}
\email{dan.timotin@imar.ro}

\begin{document}

\begin{abstract}
	Using the tools of Sz.-Nagy--Foias theory of contractions, we describe in detail the invariant subspaces of the operator $ S\oplus S^* $, where $ S $ is the unilateral shift on a Hilbert space. This answers a question of C\^amara and Ross.
\end{abstract}

\keywords{Invariant subspaces, unilateral shift, dual truncated shift}

\subjclass[2010]{47A15, 47A45, 47B37}

\maketitle

\section{Introduction}

The recent series of papers~\cite{Ding1, Ding2, Ding3} explore the class of so-called \emph{dual truncated Toeplitz operators}, which act on a subspace of the usual Lebesgue space $ L^2 $ on the unit circle $ \bbT $. In the preprint~\cite{CR} C\^amara and Ross discuss the invariant subspaces of one of these operators, the dual of the compressed shift. In their investigation they encounter the problem of determining  the invariant subspaces of the operator $ S\oplus S^* $, where   $ S $ is the usual unilateral shift operator, acting as multiplication by the variable on the Hardy-Hilbert space $ H^2 $, and they state it as an open question.

It turns out that the answer can be given through   the Sz.Nagy--Foias theory of characteristic functions of contractions on a Hilbert space~\cite{NF}. That theory includes a general result about the relation between invariant subspaces of a contraction and regular factorizations of its characteristic function. In   particular, we may use it in order to obtain an explicit description of all invariant subspaces of $ S\oplus S^* $, giving thus  a complete answer to the open question in~\cite{CR}.

The plan of the paper is the following. After some preliminaries, in Section~\ref{se:general theory} we provide a short presentation of the relevant part of the  Sz.Nagy--Foias theory. Section~\ref{se:invariant subspaces} contains the main result, the description of the invariant subspaces. Section~\ref{se:example} provides an example related to~\cite{CR}, while Section~\ref{se:parametrization} details the most interesting class of invariant subspaces.

\section{Preliminaries}\label{se:prelim}
 
 We denote shortly  $ L^2=L^2(\bbT, dm) $, where $ m $ is Lebesgue measure on the unit circle $ \bbT $. Its subspace $ H^2 $ is the Hardy-Hilbert space of functions that can be analytically extended to the unit disk $ \bbD $; then $ H^2_-=L^2\ominus H^2 $. The orthogonal projections in $ L^2 $ onto $ H^2 $ and $ H^2_- $ will be denoted by $ P_+ $ and $ P_- $ respectively.
 The map $ f\mapsto \tilde f $, with $ \tilde{f}(z)=\overline{f(\bar z)} $ is an involution on $ H^2 $.
 An inner function $ \theta\in H^2 $ is characterized by $ |\theta(e^{it})|=1 $ for almost all~$ t $. If $ \theta $ is   inner function, 
  then $ \tilde{\theta} $ is also inner.
   
   We will also use Lebesgue and Hardy spaces defined on the unit circle with values in a Hilbert space $ \EE $; they will  be denoted with $ L^2(\EE) $ and $ H^2(\EE) $ respectively.  

If $ S $ is the shift operator on $ H^2 $, defined by $ (Sf)(z)=zf(z) $, Beurling's Theorem states that  the invariant subspaces of $ S $ are $ \{0\} $ and the spaces $ \theta H^2 $ with $ \theta  $ inner. We   denote $ K_\theta=H^2\ominus \theta H^2 $; so the invariant subspaces for $ S^* $ are   $ H^2 $ and $ K_\theta $ for $ \theta $ inner. The map $ C_\theta $ defined by $ C_\theta f=\theta\bar z\bar f $ is a conjugation on $ K_\theta $; in particular,
\begin{equation}\label{eq:conjugation}
C_\theta(K_\theta)=K_\theta.
\end{equation}

It will be convenient in the sequel to consider, rather than $ S^* $, the operator $ S_* $, acting on $ H^2_- $ as the compression of multiplication by $ z $ to $ H^2_- $. This is unitarily equivalent to $ S^* $, and the precise unitary operator that implements this equivalence is $ J:H^2\to H^2_- $, $ (Jf)(z)= \bar z f(\bar z) $; we have   $ JS^*=S_*J $. The invariant subspaces of $ S_* $ are then $ H^2_- $ together with $ J(K_\theta)=J(H^2)\ominus J(\theta H^2))
=H^2_-\ominus J(\theta H^2)) $ for $ \theta $ inner.  Since
\[
\begin{split}
J(\theta H^2)&=\{J(\theta f): f\in H^2  \} =
\{ \bar z \theta(\bar z) f(\bar z): f\in H^2  \}\\
&= \{ \overline{ \tilde \theta(  z)}\bar z \overline{  \tilde f( z)}: f\in H^2  \}=\bar z\overline{\tilde{\theta}}\, \overline{H^2},
\end{split}
\]
we have
\[
 J(K_\theta)=\bar z \overline{H^2}\ominus \bar z\overline{\tilde{\theta}} \overline{H^2}
 =\bar z\overline{K_{\tilde{\theta}}}.
\]

Our purpose in this paper will be the determination of the invariant subspaces of $ S\oplus S_* $; we will see below that this is a \emph{model operator} in the sense of Sz.Nagy and Foias. The invariant subspaces of $ S\oplus S^* $ are then immediately obtained by applying the operator $ J $.
 
Some of these invariant subspaces of $ S\oplus S_* $ may easily be described; namely, the subspaces $ X\oplus X' $, where $ X\subset H^2 $ is invariant to $ S $, while $ X'\subset H^2_- $ is invariant to $S_*$. We will call them \emph{splitting} invariant subspaces.  
The next lemma  summarizes the above remarks.

\begin{lemma}\label{le:splitting subspaces}
The splitting invariant subspaces of $ S\oplus S_* $ acting on $ H^2\oplus H^2_- $ are of the form	$ X\oplus X' $, where $ X $ is either $ \{0\} $ or $ \theta H^2 $ for some inner function $ \theta $, while $ X' $ is either $ H^2_- $ or $ \bar z\overline{K_{\theta'}} $ for some inner function $ \theta' $.
\end{lemma}

One may say that these are the obvious invariant subspaces of $ S\oplus S_* $. There is, however, a large variety of nonsplitting invariant subspaces, for whose determination we will have to bring into play the Sz.-Nagy--Foias theory of contractions~\cite{NF}. 

We end the preliminaries with a lemma that will be helpful.

\begin{lemma}\label{le:formula for projection on H2-}
	If $ \theta $ is inner, then $ P_-(\bar \theta H^2)= \bar z\overline{K_{\theta}}$.	
\end{lemma}

\begin{proof}
	The decomposition $ H^2=K_\theta\oplus \theta H^2 $ yields $ \bar\theta H^2=\bar\theta K_\theta\oplus H^2 $, and so 
	$ P_-(\bar \theta H^2)=\bar\theta K_\theta $. But $\bar\theta K_\theta=\bar z\overline{K_{\theta}}  $ follows from equality~\eqref{eq:conjugation}.
\end{proof}

\section{Sz.Nagy--Foias theory of contractions and invariant subspaces}\label{se:general theory}

The general reference for   this section is the monograph~\cite{NF}.
Suppose   $ \Theta:\bbD\to  \LL(\EE, \EE^*) $ is an analytic function in the unit disc $\bbD $ with values in  the algebra of bounded operators from $ \EE $ to $ \EE^* $, with $ \|\Theta(z) \|\le 1$ for all $ z\in\bbD $; we will call it a \emph{contractive analytic function}. $ \Theta $ has boundary values almost everywhere on $ \bbT $, that  will   be denoted by $ \Theta(e^{it}) $. A contractive analytic function is called \emph{pure} if $\| \Theta(0)x\|<\|x\| $ for any $x\in \EE$.
Any contractive analytic function admits a decomposition in a direct sum $\Theta=\Theta_p\oplus \Theta_u   $, where $\Theta_p$ is pure and $\Theta_u$ is a constant unitary operator; then $\Theta_p$ is called the \emph{pure part} of $\Theta$.

To a pure contractive analytic function corresponds a functional model, defined as follows. Denote $ \Delta(e^{it})=(I-\Theta(e^{it})^*\Theta(e^{it}))^{1/2} $. Then the \emph{model space} is
\begin{equation}\label{eq:model space}
\HH_\Theta= (H^2(\EE_*)\oplus \overline{\Delta L^2(\EE)})\ominus \{ \Theta f\oplus \Delta f: f\in H^2(\EE) \},
\end{equation}
on which acts the \emph{model operator} $ \bm{S}_\Theta $, defined as the compression to $ \HH_\Theta $ of multiplication with $ e^{it} $ on both components of $ H^2(\EE_*)\oplus \overline{\Delta L^2(\EE)} $.

If $ \HH $ is a Hilbert space, a completely nonunitary contraction $ T\in \LL(\HH) $ is a linear operator that satisfies $ \|T\|\le 1 $, and there is no reducing subspace of $ T $ on which it is unitary. The defect of $T$ is the operator $ D_T=(I-T^*T)^{1/2} $, and the defect space is $\DD_T=\overline{D_T\HH} $.
It is shown in~\cite{NF} that any completely nonunitary contraction $ T $ is unitarily equivalent to $ \bm{S}_{\Theta_T} $, where $ \Theta_T $ is the pure contractive analytic function with values in $ \LL(\DD_T, \DD_{T^*}) $ defined by 
\[
\Theta_T(z)=-T+D_{T^*}(I-zT^*)^{-1}D_T|\DD_T.
\]
Note that the domain of $ \Theta_T(z) $ is $ \DD_T $.

The invariant subspaces of $\bm{S}_\Theta  $ are in correspondence with the \emph{regular} factorizations of $ \Theta $, that we will define in the sequel. Suppose $\FF  $ is a third Hilbert space and $ \Theta_1:\bbD\to \LL(\EE, \FF) $, $ \Theta_2:\bbD\to \LL(\FF, \EE_*) $ are two other contractive analytic functions such that $ \Theta=\Theta_2\Theta_1 $. If $ \Delta(e^{it})=(I-\Theta(e^{it})^*\Theta(e^{it}))^{1/2} $ for $ i=1,2 $, then the map 
\[
\Delta f\mapsto \Delta_2\Theta_1 f\oplus \Delta_1f, \quad f\in \Delta L^2(\EE)
\]
is isometric, and may thus be completed to an isometry 
\[
Z:\overline{\Delta L^2(\EE)}\to \overline{\Delta_2 L^2(\FF)}\oplus \overline{\Delta_1 L^2(\EE)}.
\]
The factorization $ \Theta=\Theta_2\Theta_1 $ is called \emph{regular} if $ Z $ is unitary.

The relation between invariant subspaces and regular factorization is summed up in the next theorem, which follows from~\cite[Theorem VII.1.1]{NF},~\cite[Theorem VII.4.3]{NF}, and the remark following it. 

\begin{theorem}\label{th:general factorization result}
	To any regular factorization $ \Theta=\Theta_2\Theta_1 $ corresponds an invariant subspace of $ \bm{S}_\Theta $, defined by the formula
	\begin{equation}\label{eq:general invariant subspace}
	\begin{split}
	Y=&\{ \Theta_2 u \oplus Z^{-1}(\Delta_2 u \oplus v) : u\in H^2(\FF), v\in  \overline{\Delta L^2(\EE)}\} \\
	&\qquad\ominus \{\Theta w\oplus \Delta w: w\in H^2(\EE) \}.
	\end{split}
	\end{equation}
The characteristic function of $\bm S_\Theta$ is the pure part of $\Theta_1$.

	Conversely, any invariant subspace determines a regular factorization $ \Theta=\Theta_2\Theta_1 $, such that $Y$ is given by~\eqref{eq:general invariant subspace}. 
	
	If $ \Theta=\Theta_2'\Theta_1' $, with $ \Theta_1:\bbD\to \LL(\EE, \FF') $, $ \Theta_2:\bbD\to \LL(\FF', \EE_*) $ produces through~\eqref{eq:general invariant subspace} the same subspace $ H_1 $, then there exists $ \Omega\in\LL(\FF,\FF') $ unitary, such that $ \Theta_1'=\Omega \Theta_1 $, $ \Theta_2'=\Theta_2\Omega^* $. 
\end{theorem}

In general, the main difficulty in the application of Theorem~\ref{th:general factorization result} is the identification of the regular factorizations of a given contractive analytic function. Fortunately, this can be done explicitely in the case that interests us.

Let us denote by $0_{m\to n}$ the zero contractive analytic function considered as acting from $\bbC^{m}$ to $\bbC^{n}$.
The functional model associated to   $\Theta=0_{1\to 1}$ is the space
\[
(H^2\oplus L^2)\ominus \{ 0\oplus f: f\in H^2\}=H^2\oplus H^2_-,
\]
 on which the model operator is precisely $ S\oplus S_* $.  
 Noting that $ \Delta(e^{it})=1 $ for all $ t $, we obtain
  \begin{equation*}
 \begin{split}
 Y&=\{ \Theta_2u\oplus Z^{-1}(\Delta_2 u\oplus v): u\in H^2(\FF), v\in \overline{\Delta_1L^2}\}\ominus (\{0\}\oplus H^2)\\
 &=P_{H^2\oplus H^2_-}( \{ \Theta_2u\oplus Z^{-1}(\Delta_2 u\oplus v): u\in H^2(\FF), v\in \overline{\Delta_1L^2}\} )\\
 &=\{ \Theta_2u\oplus P_-(Z^{-1}(\Delta_2 u\oplus v)): u\in H^2(\FF), v\in \overline{\Delta_1L^2}\},
 \end{split}
 \end{equation*}

Theorem~\ref{th:general factorization result} yields then the next corollary.

 \begin{corollary}\label{co:general} {\rm (i)}
 	To any regular factorization $ 0_{1\to 1}=\Theta_2\Theta_1 $,  where $ \Theta_1:\bbD\to\LL(\bbC,\FF) $, $\Theta_2:\bbD\to\LL(\FF,\bbC)$, corresponds an invariant subspace of $ \bm{S}_\Theta $, defined by the formula
 	  \begin{equation}\label{eq:invariant subspace}
 Y=\{ \Theta_2u\oplus P_-(Z^{-1}(\Delta_2 u\oplus v)): u\in H^2(\FF), v\in \overline{\Delta_1L^2}\},
 	\end{equation}
 	where 
\begin{equation}\label{eq:Z}
Z:L^2\to \overline{\Delta_2 L^2(\FF)}\oplus\overline{\Delta_1 L^2}, \quad
Zv=\Delta_2\Theta_1 v\oplus \Delta_1 v.
\end{equation}
The characteristic function of $T_Y:= S\oplus S_*|Y$ is the pure part of $\Theta_1$.

{\rm (ii)} 	Conversely, any invariant subspace determines a regular factorization $ 0_{1\to 1} =\Theta_2\Theta_1 $. 

{\rm (iii)}
If another factorization $ 0=\Theta_2'\Theta_1' $, with $ \Theta_1:\bbD\to\LL(\bbC,\FF') $, $\Theta_2:\bbD\to\LL(\FF',\bbC)$, produces by~\eqref{eq:invariant subspace} the same invariant subspace $ Y $, then there exists $ \Omega\in\LL(\FF, \FF') $ unitary, such that $ \Theta_1'=\Omega \Theta_1 $, $ \Theta_2'=\Theta_2\Omega^* $. 

 \end{corollary}

 In order to obtain a concrete description of the invariant subspaces in Corollary~\ref{co:general}, we have to know 
 the regular factorizations of the function $ 0_{1\to 1} $. This can be obtained from another result of Sz.-Nagy and Foias, namely~\cite[Proposition VII.3.5]{NF}, which describes all regular factorizations of a scalar contractive analytic function. Applying it to the null function yields the next statement.

 \begin{lemma}\label{le:description of factorizations}
 The regular factorizations of the function $\Theta=0_{1\to 1}$ are of the following three types, corresponding to $ \dim\FF=0,1 $ or $ 2 $:
 \begin{enumerate}
 	\item $ \dim\FF=0 $. There is a unique possibility: $0_{1\to 1}=0_{0\to 1}0_{1\to 0}$.
 	
 	\item $ \dim\FF=1 $. Then $0=\Theta_2\Theta_1$ with $ \Theta_i $ scalar functions, and there are two cases: 
 	\begin{itemize}
 		\item[(2.1)] $\Theta_1=0_{1\to 1}$, $\Theta_2$ inner.
 		
 		\item[(2.2)] $\Theta_2=0_{1\to 1}$, $\Theta_1$ inner.
 	\end{itemize}
 	
 	\item $\dim\FF=2$. Then $\Theta_1=\begin{pmatrix} \theta_{11}\\\theta_{12}
 	\end{pmatrix}$, 
 	$\Theta_2=\begin{pmatrix} \theta_{21}&\theta_{22}
 	\end{pmatrix}$, where 
 	\begin{equation}\label{eq:theta 11 and all}
 		|\theta_{11}|^2+|\theta_{12}|^2=1,
 	\quad
 	|\theta_{21}|^2+|\theta_{22}|^2=1,
 	\quad
 	\theta_{11}\theta_{21}+\theta_{12}\theta_{22}=0.
 	\end{equation}

 \end{enumerate}	
 \end{lemma}

\begin{remark}\label{re:tilde Theta}
	There is an alternate way to state conditions~\eqref{eq:theta 11 and all}: they say precisely that the $ 2\times 2 $ matrix
	\begin{equation}\label{eq:tildetheta1}
	\bm{\Theta}(e^{it}):=
	\begin{pmatrix}
	\theta_{11}(e^{it})& \bar\theta_{21}(e^{it})\\
	\theta_{12}(e^{it})& \bar\theta_{22}(e^{it})
	\end{pmatrix}
	\end{equation}
	is unitary for almost all~$ t $. Moreover, if  $\Theta'_1=\begin{pmatrix} \theta'_{11}\\\theta'_{12}
	\end{pmatrix}$, 
	$\Theta'_2=\begin{pmatrix} \theta'_{21}&\theta'_{22}
	\end{pmatrix}$ also satisfy~\eqref{eq:theta 11 and all}, then 
	$ \Theta_1'=\Omega \Theta_1 $, $ \Theta_2'=\Theta_2\Omega^* $ if and only if $ \bm{\Theta}'= \Omega \bm{\Theta} $.
\end{remark}

 \section{The invariant subspaces}\label{se:invariant subspaces}

 We may now use the information provided by Corollary~\ref{co:general} and Lemma~\ref{le:description of factorizations} in order to obtain the desired description of invariant subspaces. The next theorem is the main result of the paper.

 \begin{theorem}\label{th:main result}
 	The invariant subspaces of $ S\oplus S_* $ acting on $ H^2\oplus H^2_- $ are the following:
 	
 	\begin{itemize}
 		\item[\rm (I)] Splitting invariant subspaces; that is,
 		\[
 		Y=X\oplus X' 
 		\] 
 		with $ X\subset H^2 $ is invariant to $ S $,  $ X'\subset H^2_- $ is invariant to $S_*$.
 		
 		\item[\rm (II)] Nonsplitting invariant subspaces. These are of the form
 		\begin{equation}\label{eq:nonsplitting subspaces}
 			Y=
 		\{ (\theta_{21} u_1+\theta_{22} u_2)\oplus P_- (\bar\theta_{11} u_1+\bar\theta_{12} u_2) : u_1, u_2\in H^2\},
 		\end{equation}
 		where $ \theta_{ij} $ are  functions in the unit ball of $ H^\infty $, such that $ \theta_{11} $ and $ \theta_{12} $ are not proportional, and the matrix  
 		\begin{equation}\label{eq:tilde Theta}
 		\bm{\Theta}(e^{it}):=
 		\begin{pmatrix}
 		\theta_{11}(e^{it})& \bar\theta_{21}(e^{it})\\
 		\theta_{12}(e^{it})& \bar\theta_{22}(e^{it})
 		\end{pmatrix}
 		\end{equation}
 		is unitary almost everywhere (equivalently, $ \theta_{ij} $ satisfy~\eqref{eq:theta 11 and all}).
 		
 		Two matrices $ \bm{\Theta}, \bm{\Theta}' $ define the same invariant subspace if and only if there exists a unitary $ 2\times 2 $ matrix $ \Omega $ with scalar entries, such that $ \bm{\Theta}=\Omega \bm{\Theta}' $.
 	\end{itemize}
 	
 \end{theorem}

 \begin{proof}

 	We take one by one the possibilities displayed in the statement of Lemma~\ref{le:description of factorizations}.
 	
 \subsection*{Case (1)}	  We have $\dim \FF=0$, so $ \Delta_2=0 $, and $ \Delta_1 = I_{L^2}$, so 
 	$Z:L^2\to \{0\}\oplus L^2$, $Z(v)=0\oplus v$.
 	\[
 	Y=\{0\oplus v: v\in L^2\}\ominus (\{0\}\oplus H^2)=\{0\}\oplus H^2_-.
 	\]
 	The invariant subspace $Y$ is the second component (the space on which acts $S^*$); it is obviously splitting.

 \subsection*{Case (2.1)} Here $\dim\FF=1$, $\Delta_1=I_{L^2}$, $\Delta_2=0$.
 $Z$ is the same operator as in the previous case. We have
 \[
 Y=\{\Theta_2 u\oplus v: u\in H^2, v\in L^2\}\ominus (\{0\}\oplus H^2)=\Theta_2 H^2\oplus H^2_-.
 \]

 \subsection*{Case (2.2)}
 Again  $\dim\FF=1$, but $\Delta_1=0$, $\Delta_2=I_{L^2}$.
 So $Z:L^2\to L^2\oplus \{0\}$, $Zv=\Theta_1 v\oplus 0$, $ Z^{-1}(w\oplus 0)=\bar{\Theta_1}w $. Then
 \[
 Y=\{0\oplus \bar\Theta_1 v: v\in H^2\} \ominus (\{0\}\oplus H^2).
 \]
 Since the projection of $ \bar{\Theta}_1 H^2 $ onto $ H^2_- $ is $ \bar \Theta_1 K_{\Theta_1} $, whence
 \[
Y =\{0\}\oplus \bar \Theta_1 K_{\Theta_1}. 
 \]
 
 One sees that both cases (2.1) and (2.2) lead to splitting invariant subspaces.

 \subsection*{Case (3)}
Here we have $ \dim\FF=2 $. From~\eqref{eq:theta 11 and all} it follows that  $\Theta_1^* \Theta_1=  \Theta_2 \Theta_2^*=I_{\bbC^2} $ a.e., so $ \Delta_1=0 $, while $ \Delta_2 $ is a projection a.e.; that is, $ \Delta_2=\Delta_2^2 $. Also, $ \bm{\Theta} $ unitary a.e. implies that $ \bm{\Theta}\bm{\Theta}^*=I_{\bbC^2} $ a.e, which is equivalent to
    \[
    \Theta_1\Theta_1^*+\Theta_2^*\Theta_2=I.
    \]
       Therefore
     \[
     \Delta_2\Theta_1= \Delta_2^2\Theta_1 = (I-\Theta_2^*\Theta_2)\Theta_1=\Theta_1\Theta_1^*\Theta_1=\Theta_1.
     \]     
     It follows that $ Z:L^2\to \Delta_2 L^2\subset L^2(\bbC^2) $ is defined by $ Zw=\Theta_1 w $, and $ Z^{-1}(\Delta_2 u) =Z^{-1}(\Delta_2^2 u)=Z^{-1}(\Theta_1\Theta_1^* u)=\Theta_1^*u $.

     Denoting   $ u\in H^2(\bbC^2) $ by $ u=\begin{pmatrix} u_1\\ u_2
     \end{pmatrix} $, we have, according to~\eqref{eq:invariant subspace},
     \begin{equation}\label{eq:case 3}
     Y=
     \{ (\theta_{21} u_1+\theta_{22} u_2)\oplus P_- (\bar\theta_{11} u_1+\bar\theta_{12} u_2) : u\in H^2(\bbC^2)\}.
   \end{equation}

    We have again to discuss two cases.  
    \subsection*{Case (3.1)}
    Suppose $ \theta_{11}, \theta_{12} $ are proportional; then from~\eqref{eq:theta 11 and all}  it follows that $\theta_{11}=\alpha_1 \theta $, $ \theta_{12}=\alpha_2 \theta $ for some inner function $ \theta $ and complex numbers $ \alpha_i $ with $ |\alpha_1|^2+|\alpha_2|^2=1 $. Then again by~\eqref{eq:theta 11 and all} we have $ \alpha_1\theta_{21}+\alpha_2\theta_{22}=0 $. Therefore $\bm{\Theta}  $ defined by~\eqref{eq:tilde Theta} is
    \[
    \bm{\Theta}=\begin{pmatrix}
    \alpha_1\theta& \bar{\theta}_{21}\\
    \alpha_2\theta& \bar \theta_{22}
    \end{pmatrix}.
    \]
    The $ 2\times 2 $ matrix $ \Omega=\begin{pmatrix}
    \bar{\alpha}_1&\bar{\alpha}_2\\ 
    \alpha_2&-\alpha_1
    \end{pmatrix} $
    is unitary, and 
    \[
    \Omega\bm{\Theta}=
    \begin{pmatrix}
    \theta & 0\\0 &\alpha_2\bar\theta_{21}-\alpha_1\bar\theta_{22}
    \end{pmatrix}.
    \]
    Since $ \Omega\bm{\Theta} $ must be unitary almost everywhere, it follows that $ \theta':= \bar \alpha_2\theta_{21}-\bar\alpha_1\theta_{22}$ is inner.
   By Corollary~\ref{co:general} (iii) and Remark~\ref{re:tilde Theta}, the invariant subspace obtained in~\eqref{eq:case 3} can also be defined by  $\bm{\Theta}'= \Omega\bm{\Theta} $, and so
    \begin{equation}\label{eq:Yfor case 3.1}
      Y=\{\theta'u_2\oplus\bar{\theta}u_1: u_1, u_2\in H^2 \} = 
    \theta'H^2\oplus \bar{\theta} K_{\theta}.
    \end{equation}
    
    Note that this case completes the list of splitting invariant subspaces described in Lemma~\ref{le:splitting subspaces}.

    \subsection*{Case (3.2)}
    The last case appears when   $\theta_{11}$ and $ \theta_{12} $ are not proportional, which leads us precisely to the invariant subspaces of type (II). To finish the proof, we have to show that these do not split. We already know  the subspaces that split, so we must show that subspaces of type (II) do not coincide with any of them. By Corollary~\ref{co:general} (iii) any other factorization producing the same invariant subspace must be of type (3), with the associated invariant subspace given by~\eqref{eq:Yfor case 3.1}. So we must have two inner functions $ \theta, \theta' $ and a $ 2\times 2 $ unitary matrix $ \begin{pmatrix}
    a_{11}& a_{21}\\a_{12}& a_{22}
    \end{pmatrix} $ 
    such that
    \[
     \begin{pmatrix}
    \theta_{11}&\bar  \theta_{21}\\ \theta_{12}& \bar\theta_{22}
    \end{pmatrix} =
     \begin{pmatrix}
    a_{11}& a_{21}\\a_{12}& a_{22}
    \end{pmatrix} 
     \begin{pmatrix}
    \theta& 0\\0& \bar\theta'
    \end{pmatrix} .
    \]
    It follows that $ \theta_{11}=a_{11}\theta $, $ \theta_{12}=a_{12}\theta $. So $ \theta_{11} $ and $ \theta_{12} $ are proportional, contrary to the assumption. The proof is thus finished.
     \end{proof}

\begin{remark}\label{re:about subspaces} (i)
An alternate compact way to write the nonsplitting subspaces in~\eqref{eq:nonsplitting subspaces} can be obtained if we consider $ \bm{\Theta} $ as a multiplication operator on $ L^2\oplus L^2 $. Then 
\[
Y=(P_+\oplus P_-)\bm{\Theta}^*(H^2\oplus H^2).
\]

(ii)  The case (I) of Theorem~\ref{th:main result}, that is, the identification of all splitting subspaces, is stated in~\cite[Theorem 8.1]{CR}.
\end{remark}

\begin{remark}\label{re:which are pure}
	As noted in Corollary~\ref{co:general}, the characteristic function of $ T_Y $ is the pure part of $ \Theta_1 $. If we examine the factorizations in Lemma~\ref{le:description of factorizations}, we see that 
	  $ \Theta_1 $ is pure in cases (1) and (2.1). 
	  In case (2.2), $ \Theta_1 $ is pure if it is nonconstant.
	  In case (3), $ \Theta_1 $ is pure, except when $ \theta_{11}=t_1 $ and $ \theta_{12}=t_2 $ are constant scalars satisfying 	$|t_1|^2+|t_2|^2=1$. 
	  Consequently, in all these cases the characteristic function of $ T_Y $ is $ \Theta_1 $ and $ \dim\DD_{T_Y}=1 $.

	The remaining  cases, which lead to $ \dim\DD_{T_Y}=0 $, are then:
	\begin{itemize}
		\item (2.2), with $ \Theta_1 $ constant of modulus~1. The invariant subspace is $ \{0\} $.
		
		\item  (3), with  $ \theta_{11}=t_1 $ and $ \theta_{12}=t_2 $  constant scalars satisfying 	$|t_1|^2+|t_2|^2=1$. 
		Arguing as in the proof of~\ref{th:main result}, Case 3.1, we obtain a $ 2\times 2 $ unitary matrix $ \Omega $ and an inner function $ \theta' $ such that
		\[
		\Omega\bm{\Theta}=
		\begin{pmatrix}
		1&0\\0&\theta'
		\end{pmatrix}.
		\] 
		The invariant subspace is $ Y=\theta'H^2\oplus \{0\} $, and the characteristic function of $ T_Y $ is the pure part of $ \Omega\Theta_1=\begin{pmatrix}
		1\\0
		\end{pmatrix} $, which is $ 0_{0\to 1} $. 
		
	\end{itemize}

\end{remark}

  Using Remark~\ref{re:which are pure}, we may identify the reducing subspaces of $ S\oplus S_* $.

\begin{theorem}\label{th:reducing}
	The only reducing subspaces for $ S\oplus S_* $ are $ H^2\oplus\{0\} $ and $ \{0\}\oplus H^2_- $.
\end{theorem}

\begin{proof}
	Suppose $ H^2\oplus H^2_-=Y_1+ Y_2 $, with $ Y_1\perp Y_2 $ nontrivial and both invariant with respect to $ S\oplus S_* $. Then $ 1=\dim \DD_{S\oplus S_*}= \dim\DD_{T_{Y_1}}+ \dim\DD_{T_{Y_2}} $.  
	We may assume that $ \dim\DD_{T_{Y_1}}=0$, $\dim\DD_{T_{Y_2}}=1 $. It follows from Remark~\ref{re:which are pure} that either $	Y_1=\{0\} $ or $Y_1=\theta_1 H^2\oplus \{0\} $ with $ \theta_1 $ inner. The nontrivial case is the latter; then $ Y_2=K_{\theta_1}\oplus H^2_- $ is invariant only if $ \theta_1\equiv 1 $. 
\end{proof}

    \section{An example}\label{se:example}
 
 \newcommand{\Beta}{(1-|\alpha|^2)^{1/2}} 
 
 The following example exhibits a whole class of nonsplitting invariant subspaces, for which we may obtain a simpler form than that given by Theorem~\ref{th:main result}.
 
 \begin{example}
 	Fix $ \alpha,    a\in \bbD $, and $ \alpha\not=0 $.
 	Denote $ b_a(z)=\frac{z-a}{1-\bar a z} $, and define
 	the functions $ \theta_{ij} $, $ i,j=1,2 $,  by
 	\[
 	\theta_{11}(z)=\theta_{22}(z)=\alpha b_a(z),\quad
 	\theta_{12}(z)=-\Beta, \quad	\theta_{21}(z)= (1-|\alpha|^2)^{1/2}.
 	\]
 	By Theorem~\ref{th:main result}, we obtain  the following nonsplitting subspace:
 	\[
 	Y=\{ (\Beta u_1+\alpha b_a u_2)\oplus P_-(\overline{\alpha b_a}u_1-\Beta u_2): u_1, u_2\in H^2  \}.
 	\]
 	There is a simpler way to write this subspace. First, $ P_- u_2=0 $. Secondly, $ K_{b_a} $ is a one dimensional space generated by the reproducing kernel $ k_a(z)=\frac{1}{1-\bar a z} $, and the orthogonal projection onto $ K_{b_a} $ has the formula $P_{K_{b_a} }f=(1-|a|^2)f(a)k_a $. We have the orthogonal decomposition $ H^2=K_{b_a}\oplus b_aH^2 $, according to which    $ u_1=(1-|a|^2)u_1(a)k_a+b_a u_1'  $. Therefore 
 	\[
 	\Beta u_1+\alpha b_a u_2=\Beta(1-|a|^2)u_1(a)k_a+ b_a(\Beta u_1'+\alpha u_2)
 	\]
 	and 
 	\[
 	\begin{split}
 	P_-(\overline{\alpha b_a}u_1-\Beta u_2)&=P_-(\overline{\alpha b_a}(1-|a|^2)u_1(a)k_a+\bar\alpha u_1'-\Beta u_2)\\&=P_-(\overline{\alpha b_a}(1-|a|^2)u_1(a)k_a)
 	= \frac{\bar\alpha(1-|a|^2)u_1(a) \bar z}{1-a\bar z}.
 	\end{split}
 	\]	
 	If we denote $ u=\Beta u_1+\alpha b_a u_2 $, then $ u(a)=\Beta u_1(a) $; moreover, if $ u_1, u_2 $ are arbitrary functions in $ H^2 $, then $ u $ is also an arbitrary function in $ H^2 $. We may therefore write
 	\begin{equation*} 
 	Y=\{ u\oplus  \frac{\bar \alpha(1-|a|^2)u(a) \bar z}{\Beta (1-a\bar z)} : u\in H^2 \}.
 	\end{equation*}
 	It is easy to see that when $ \alpha\in\bbD\setminus\{0\} $, $ \frac{\bar \alpha(1-|a|^2)}{\Beta} $ covers $ \bbC\setminus\{0\} $. Let us denote $ \beta=  \frac{\bar \alpha(1-|a|^2)}{\Beta} $; the invariant subspace is then 
 	\begin{equation}\label{eq:example}
 	Y=\{ u\oplus\beta \frac{ u(a) \bar z}{1-a\bar z} : u\in H^2 \}.
 	\end{equation}

 	We have thus obtained in~\eqref{eq:example} 
 	a class of nonsplitting invariant subspaces parameterized by the nonzero complex number $ \beta $. For an appropriate value of this parameter,
 	$ Y $ corresponds to the subspace appearing in Example 7.3 of~\cite{CR} (after taking   into account the unitary equivalence implemented by $ J:H^2\to H^2_- $).
 	
 \end{example}

     \section{Parametrization of nonsplitting subspaces}\label{se:parametrization}
     
     The nonsplitting subspaces are the most interesting ones, so   it is worth to obtain a more detailed description of this class.
     Equations~\ref{eq:theta 11 and all} define $ \Theta_1 $ and $ \Theta_2 $ in an implicit manner; we will determine in this last section a parametrization of these two functions.
     
      We start with a pair of nonproportional functions $\theta_{11}, \theta_{12}\in H^\infty$ that satisfy $   |\theta_{11}|^2+|\theta_{12}|^2=1$.
     First, if we denote by $g_1, g_2$ the outer parts of $ \theta_{11}, \theta_{12} $ respectively, they   satisfy $|g_1|^2+|g_2|^2=1$. In fact, this means an  outer function $ g_1 $ bounded by 1 and  subject to the condition $ \int (1-|g_1|^2) >-\infty $, which is equivalent to $ g_1 $ not being an extreme point of the unit ball of $ H^\infty $ (see~\cite{LR}). Then $ g_1 $ determines   $ g_2 $ up to a scalar of modulus~1.
     
     Since $ \bm{\Theta} $ is unitary, $ |\theta_{11}(e^{it})|=|\theta_{22}(e^{it})| $ and $|\theta_{12}(e^{it})|=|\theta_{21}(e^{it})| $ almost everywhere. Therefore $g_1, g_2$ are also the outer parts of $ \theta_{22}, \theta_{21} $ respectively. We may then write
     \[
     \theta_{11}=\alpha_{11}g_1,\quad \theta_{12}=\alpha_{12}g_2, \quad
     \theta_{21}=\alpha_{21}g_2, \quad
     \theta_{22}=\alpha_{22}g_1,
     \]
     with $ \alpha_{ij} $ inner; from the last formula in~\eqref{eq:theta 11 and all} it follows that
     \[
     	\alpha_{11}\alpha_{21}+\alpha_{12}\alpha_{22}=0.
     \]
     By factoring common inner divisors, let us then write $ \alpha_{ij}=\alpha_i\beta_{ij} $, with $ (\beta_{i1}, \beta_{i2})=1 $. It follows then that
     \[
     	\beta_{11}\beta_{21}=-\beta_{12}\beta_{22}.
     \]
     Divisibility implies then that $ \beta_{22}=\lambda \beta_{11} $ and $ \beta_{21}=-\lambda \beta_{12} $ for some $ \lambda\in\bbC $, $ |\lambda|=1 $. If we denote, for simplicity, $ \beta_1=\beta_{11} $ and $ \beta_2=\beta_{12} $,
     we may write 
     \[
     Y=\{ (\lambda\alpha_2 [-\beta_2g_2 u_1+\beta_1g_1 u_2])
     \oplus P_-(\bar\alpha_1[\bar\beta_1 \bar g_1 u_1+\bar\beta_2\bar g_2 u_2]) \}.
     \]
     
     
     So the nonsplitting invariant subspace  $Y$ is determined by the following ``free'' objects:
     \begin{itemize}
     	\item[(i)]  
     	An  outer function $ g_1 $ bounded by 1 that is not  an extreme point of the unit ball of $ H^\infty $.

     	\item[(ii)] Two arbitrary inner functions $\alpha_1, \alpha_2$.
     	
     	\item[(iii)] Two arbitrary, but coprime inner functions $\beta_1, \beta_2$.
     	
     	\item[(iv)] A complex number $ \lambda $ of modulus 1.
     	
     \end{itemize}
     To obtain from these parameters $ \theta_{ij} $, note first that $ g_1 $ determines up to a constant of modulus 1 an outer function $ g_2 $, such that $ |g_1|^2+|g_2|^2=1 $. Then we have
     \begin{equation}\label{eq:parametrization}
         \theta_{11}=\alpha_{1}\beta_1 g_1,\quad \theta_{12}=\alpha_{1}\beta_2 g_2, \quad
     \theta_{21}=-\lambda\alpha_{2}\beta_2g_2, \quad
     \theta_{22}=\lambda\alpha_{2}\beta_1g_1.
     \end{equation}
     
The condition for two parametrizations to produce the same invariant subspace follows from the last statement of Theorem~\ref{th:main result}.   One sees that there is a remarkable richness of nonsplitting subspaces.

\end{document}